\pdfoutput=1
\documentclass[preprint,11pt,a4paper]{elsarticle}

\usepackage[margin=25mm]{geometry}
\usepackage[T1]{fontenc}
\usepackage{lmodern}
\usepackage{amsmath,amssymb,amsthm}
\usepackage{bm}
\usepackage{siunitx}
\usepackage{booktabs}
\usepackage{array}
\usepackage{graphicx}
\usepackage{subcaption}
\usepackage{xcolor}
\usepackage{microtype}
\usepackage[hidelinks]{hyperref}
\usepackage[capitalise,nameinlink]{cleveref}
\hypersetup{pdfauthor={Wyatt Horne},pdftitle={Swarmlet: a GPU-resident reformulated vortex particle method for
  coupled multirotor aerodynamics and flight dynamics}}

\biboptions{sort&compress}

\makeatletter
\def\ps@pprintTitle{\let\@oddhead\@empty\let\@evenhead\@empty
  \def\@oddfoot{\footnotesize\itshape Preprint\hfill October 2026}\let\@evenfoot\@oddfoot}
\makeatother

\newtheorem{proposition}{Proposition}
\newcommand{\vect}[1]{\boldsymbol{#1}}
\newcommand{\dd}{\mathrm{d}}
\newcommand{\xp}{\vect{x}_p}
\newcommand{\ap}{\vect{\alpha}_p}
\newcommand{\uvec}{\vect{u}}
\newcommand{\wvec}{\vect{\omega}}
\newcommand{\Jt}{\mathsf{J}}
\newcommand{\order}{\mathcal{O}}

\graphicspath{{figures/}}

\begin{document}

\begin{frontmatter}

\title{Swarmlet: a GPU-resident reformulated vortex particle method for coupled multirotor aerodynamics and flight dynamics}

\author[inst1]{Wyatt Horne\corref{cor1}}
\ead{whorne@sandia.gov}
\cortext[cor1]{Corresponding author.}
\affiliation[inst1]{organization={Sandia National Laboratories},
  city={Albuquerque},
  country={USA}}

\begin{abstract}
We present Swarmlet, a GPU solver for the unsteady aerodynamics and flight dynamics of one or many multirotor
vehicles. The wake is a meshless, evolving-core vortex particle field in the reformulated (rVPM) large-eddy
formulation, coupled to rotor models of three fidelities, six-degree-of-freedom vehicles and flight control.
Particles carry a smooth core of compact support whose field outside four core radii is exactly the singular
Biot--Savart field, so the multipole far field of the treecode carries no regularization error. The subfilter closure
is a dynamic eddy viscosity on a fixed filter width, applied by core spreading with a Lagrangian exchange of its
local excess under an explicit stability bound, and the strengths are relaxed by projection; global corrections make
both the exchange and the relaxation conserve total vorticity and linear impulse. All rotor models shed connected
bound circulation, the disks along the arcs swept in each step however long, so the wake receives the lift impulse. The solver agrees with direct summation to $5\times10^{-7}$ and
converges at third order in time. Without calibration, an actuator-line model that includes its own near sheet
predicts the hover thrust of a DJI~9443 propeller within $2\%$. With calibrated section polars it reproduces a
Crazyflie propeller's thrust and torque within $3.2\%$ at three speeds and the near wake of a hovering quadrotor
with relative $L^2$ differences of \numrange{0.08}{0.21} against particle image velocimetry, while the wake keeps the
delivered impulse. A closed-loop formation of 100 quadrotors, run at a coarser setting on one GPU, illustrates the
coupled method at swarm scale.
\end{abstract}


\end{frontmatter}

\section{Introduction}
\label{sec:introduction}

Small multirotor vehicles increasingly fly close to one another, to terrain and to structures. Their wakes interact:
a lower vehicle flies in the jet of its neighbour, rotors on one airframe ingest each other's tip vortices, and the
resulting changes in thrust and moment must be absorbed by the flight
controller~\cite{yoon2016,shukla2018,alvarez2020,lee2020,kiran2024}. Predicting these effects requires an unsteady
wake model coupled both to the rotor loads that create the wakes and to the motion of the vehicles. Blade-element
momentum theory assumes an isolated rotor in steady, axisymmetric inflow, and mesh-based large-eddy simulation must
transport the wakes through a volume mesh that follows every vehicle, which is expensive for one rotorcraft and
prohibitive for a formation.

Lagrangian vortex methods occupy the middle ground~\cite{leonard1980,cottet2000}. Particles carry vorticity, so
computational elements exist only where vorticity does and are convected without numerical dissipation. Regularized
vortex particles have a convergence theory~\cite{beale1982,cottet2000}, viscous diffusion can be represented by
particle strength exchange~\cite{degond1989} or core spreading, and large simulations have been performed on
parallel machines and GPUs~\cite{chatelain2008,rossinelli2010,yokota2013}. For rotorcraft, the viscous vortex
particle method of He and Zhao~\cite{he2009}, DUST~\cite{tugnoli2021} and FLOWUnsteady~\cite{alvarez2022flowunsteady}
couple particle wakes to lifting-line or panel models of the rotors at a fraction of the cost of mesh-based CFD. Such
wakes become numerically unstable once they turn turbulent, because stretching inflates particle strengths while the
cores stay fixed. The reformulated vortex particle method (rVPM) of Alvarez and Ning~\cite{alvarez2024} resolves this
by sharing stretching between strength and core size and adding a subfilter-scale (SFS) model, which yields a
meshless large-eddy simulation.

This paper describes Swarmlet, a GPU solver built on rVPM for the aerodynamics and flight dynamics of one or many
multirotor vehicles. The same method runs on one GPU or on many, and one set of
wake numerics is used for every validation case. Its contributions are:
\begin{enumerate}
\item a smooth vorticity core of compact support whose exterior field is exactly the singular Biot--Savart field, so
that the multipole far field carries no regularization error and near-field operators have finite range
(\cref{sec:core});
\item an SFS closure for rVPM with a dynamic eddy viscosity on a fixed filter width, applied by core spreading with a
Lagrangian exchange of its local excess, and a relaxation by projection; the changes of total vorticity and linear
impulse of both the exchange and the relaxation are removed over the whole wake (\cref{sec:sfs,sec:relaxation});
\item rotor models of three fidelities that shed their lift as connected bound circulation, so that the wake receives
the lift impulse, including an actuator line with the induction of its own near sheet, disks that shed along
the arcs they sweep at any time step, and a blade-element disk with Prandtl's tip loss on its inflow
(\cref{sec:rotors});
\item a single GPU-resident treecode for one or many MPI ranks, coupled to six-degree-of-freedom vehicles and flight
control (\cref{sec:coupling,sec:solver});
\item verification on canonical problems and validation against three experiments
(\cref{sec:verification,sec:validation}), and a demonstration on a closed-loop formation of 100 quadrotors
(\cref{sec:swarm}).
\end{enumerate}
\Cref{sec:method} presents the method, \cref{sec:verification,sec:validation} its verification and validation and
\cref{sec:swarm} the swarm demonstration; \cref{sec:discussion} discusses limitations and concludes.

\section{Method}
\label{sec:method}

\subsection{Filtered vorticity transport and vortex particles}
\label{sec:formulation}

For incompressible flow of density $\rho$ and kinematic viscosity $\nu$, the filtered vorticity
$\wvec=\nabla\times\uvec$ satisfies
\begin{equation}
\frac{\mathrm{D}\wvec}{\mathrm{D}t}=(\wvec\cdot\nabla)\uvec+\nu\nabla^2\wvec+\vect{E}_{\mathrm{sfs}},\qquad
\uvec(\vect{x})=\uvec_\infty+\frac{1}{4\pi}\int\frac{\wvec(\vect{y})\times(\vect{x}-\vect{y})}{|\vect{x}-\vect{y}|^3}\,\dd\vect{y}
+\uvec_{\mathrm{b}},
\label{eq:vorticity}
\end{equation}
where $\vect{E}_{\mathrm{sfs}}$ is the effect of unresolved scales, $\uvec_\infty$ a prescribed background flow and
$\uvec_{\mathrm{b}}$ the field of the circulation bound to rotor blades (\cref{sec:rotors}). All flows considered
here are unbounded. The free vorticity is carried by $N$ particles with position $\xp$, vector strength $\ap$,
core size $\sigma_p$ and material volume $V_p$, $\wvec_\sigma=\sum_p\ap\,\sigma_p^{-3}\zeta(|\vect{x}-\xp|/\sigma_p)$,
with a radially symmetric core $\zeta$ of unit integral. The induced velocity is
\begin{equation}
\uvec_\sigma(\vect{x})=\sum_{p=1}^N q(r_p/\sigma_p)\,\frac{\ap\times\vect{r}_p}{4\pi r_p^3},\qquad
\vect{r}_p=\vect{x}-\xp,\qquad q(\rho)=4\pi\int_0^\rho\zeta(s)s^2\,\dd s,
\label{eq:particle-velocity}
\end{equation}
and its gradient $\Jt=\nabla\uvec$ follows analytically. Stretching is discretized in the transposed form
$\vect{S}_p=\Jt_p^{\mathsf T}\ap$~\cite{winckelmans1993,cottet2000}, with $\Jt_p$ the gradient of the total field.
In the reformulated vortex particle method (rVPM) of Alvarez and Ning~\cite{alvarez2024} the stretching along a
particle's own axis is shared between its strength and its core; with their coefficients $f=0$, $g=1/5$ and
$\hat\ap=\ap/|\ap|$,
\begin{equation}
\frac{\dd\xp}{\dd t}=\uvec(\xp),\qquad
\frac{\dd\ap}{\dd t}=\vect{S}_p-\frac35\left(\vect{S}_p\cdot\hat\ap\right)\hat\ap,\qquad
\frac{\dd\sigma_p}{\dd t}=-\frac15\,\sigma_p\,\frac{\vect{S}_p\cdot\hat\ap}{|\ap|},\qquad
\frac{\dd V_p}{\dd t}=0,
\label{eq:rvpm}
\end{equation}
to which the viscous and subfilter terms of \cref{sec:sfs} are added. The core is integrated as $\sigma_p^2$.

\subsection{A smooth core of compact support}
\label{sec:core}

The common Gaussian core~\cite{winckelmans1993,alvarez2024}, $\zeta_G(\rho)=(2\pi)^{-3/2}e^{-\rho^2/2}$, differs
from the singular kernel at every distance, so a fast summation method must neglect or correct that difference. We
subtract from the Gaussian its first-order Taylor expansion in $\rho^2$ about $\rho^2=R^2$ and renormalize,
\begin{equation}
\tilde\zeta(\rho)=K(2\pi)^{-3/2}\left[e^{-\rho^2/2}-\gamma\left(1+\frac{R^2-\rho^2}{2}\right)\right]\ \ (\rho<R),\qquad
\tilde\zeta=0\ \ (\rho\ge R),\qquad \gamma=e^{-R^2/2},
\label{eq:compact-core}
\end{equation}
with $R=4$ and $K=1.02576$. Both $\tilde\zeta$ and $\tilde\zeta'$ vanish at $\rho=R$, so the velocity and its
gradient are continuous, and the polynomial term integrates in closed form,
$\tilde q(\rho)=K[q_G(\rho)-\sqrt{2/\pi}\,\gamma\,\rho^3((1+R^2/2)/3-\rho^2/10)]$ for $\rho<R$, with
$q_G(\rho)=\operatorname{erf}(\rho/\sqrt2)-\sqrt{2/\pi}\,\rho e^{-\rho^2/2}$. The compact core has a $2.3\%$ higher
peak and a $2.2\%$ smaller root-mean-square radius than the Gaussian of the same $\sigma$ (\cref{fig:core}).

\begin{proposition}[Exact exterior field]
\label{prop:exterior}
If $\zeta$ is radially symmetric with unit integral and support in the ball of radius $R$, then at every point with
$|\vect{x}-\xp|\ge R\sigma_p$ the velocity and velocity gradient induced by particle $p$ equal those of a singular
particle of strength $\ap$ at $\xp$.
\end{proposition}
\begin{proof}
The velocity is the curl of $\vect\psi=\sum_p\ap\phi_p$ with $-\nabla^2\phi_p=\sigma_p^{-3}\zeta(r/\sigma_p)$. As
$\phi_p$ is radial, the divergence theorem gives $\phi_p(r)=\int_r^\infty q(s/\sigma_p)/(4\pi s^2)\,\dd s$, which
for $r\ge R\sigma_p$ is the singular potential $1/(4\pi r)$ (Newton's shell theorem).
\end{proof}

Hence a multipole expansion evaluated outside all supports of its particles carries no regularization error, and
the directly summed interactions have finite range (\cref{sec:solver}). The price is a slightly different core
shape: a thin ring translates $0.44\%$ faster than with a Gaussian core (\cref{sec:ver-ring}).

\begin{figure}[t]
\centering
\includegraphics[width=\textwidth]{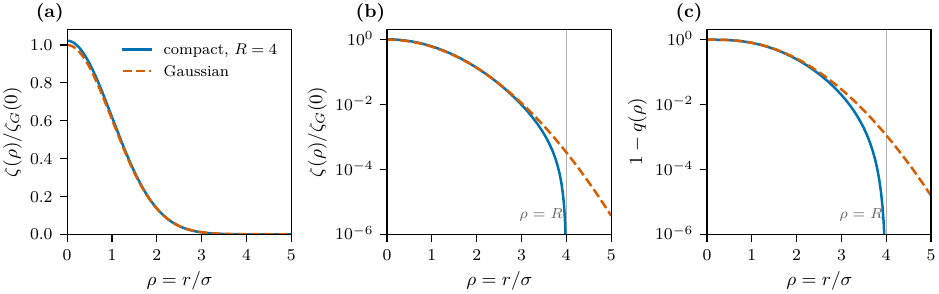}
\caption{The compact core, \cref{eq:compact-core} with $R=4$ (solid), and the Gaussian core of the same $\sigma$
(dashed). (a) Radial profile $\zeta(\rho)$; (b) the same on a logarithmic scale; (c) the fraction $1-q(\rho)$ of the
strength outside radius $\rho\sigma$, the relative difference between the regularized and the singular kernel,
which vanishes beyond $\rho=R$ for the compact core.}
\label{fig:core}
\end{figure}

\subsection{Subfilter closure}
\label{sec:sfs}

Large-eddy vortex methods model $\vect{E}_{\mathrm{sfs}}$ as a dissipation of strength along the stretching
direction~\cite{alvarez2024} or as an eddy-viscosity diffusion~\cite{mansfield1998,winckelmans1993}. We use the latter,
$\vect{E}_{\mathrm{sfs}}=\nabla\cdot(\nu_T\nabla\wvec)$ with $\nu_T=C\Delta^2|\mathsf{S}|$, $\mathsf{S}$ the
strain-rate tensor and $|\mathsf{S}|=(2\mathsf{S}:\mathsf{S})^{1/2}$. The filter width is fixed at the resolution
of the field, $\Delta^2=12\sigma_{s,\max}^2$, the box-filter width whose second moment equals that of a Gaussian of
the largest source core $\sigma_{s,\max}$ of the rotors (\cref{sec:rotors})~\cite{pope2000}. It does not follow the particle cores: under stretching a core
shrinks, and a filter width proportional to it would switch off the model of the particles that need it most. The
coefficient is obtained by the dynamic procedure~\cite{germano1991} in Lilly's least-squares form~\cite{lilly1992}
with a test filter of twice the width. Replacing the Leonard and model tensors by their leading-order
expansions~\cite{clark1979}, $\mathsf{L}\approx\frac{\Delta^2}{4}\mathsf{G}$ with $\mathsf{G}=\Jt\Jt^{\mathsf T}$ and
$\mathsf{M}\approx-6\Delta^2|\mathsf{S}|\mathsf{S}$, gives the global coefficient
\begin{equation}
C=\max\left(0,\;-\frac{1}{24}\,
\frac{\langle|\mathsf{S}|\,\mathsf{G}:\mathsf{S}\rangle}{\langle|\mathsf{S}|^2\,\mathsf{S}:\mathsf{S}\rangle}\right),
\label{eq:dynamic-coefficient}
\end{equation}
where $\langle\cdot\rangle$ is a volume-weighted average over all particles, evaluated at every right-hand side.
Because $\mathsf{G}:\mathsf{S}=\operatorname{tr}(\mathsf{S}^3)-\frac14\wvec\cdot\mathsf{S}\wvec$ contains only strain
self-amplification and vortex stretching, $C$ vanishes for two-dimensional laminar vortices.

The global viscosity is applied together with the molecular one by core spreading~\cite{leonard1980},
\begin{equation}
\frac{\dd\sigma_p^2}{\dd t}=2\,\bigl(\nu+C\Delta^2|\mathsf{S}_p|\bigr),
\label{eq:core-spreading}
\end{equation}
added to the rVPM core rate. Strengths and positions are untouched, so total vorticity and linear impulse are
conserved exactly and no overlapping neighbours are needed. Grown cores convect with the velocity at their centre,
the consistency limit of core spreading~\cite{greengard1985}; in the hover wake of \cref{sec:val-wake} the median
core of particles \numrange{24}{48} revolutions old is $1.17\sigma_s$.

Core spreading never removes strength, whereas under \cref{eq:rvpm} strengths grow with the stretching rate
$s=\hat\ap\cdot\Jt\hat\ap$ as $\dd\ln|\ap|/\dd t=\frac25s$, so that in developed turbulence strong particles of
opposite sign accumulate. The global coefficient, moreover, is dominated by the laminar shear layers and tip vortices
of a rotor wake. Each particle therefore carries a Lagrangian-averaged coefficient~\cite{meneveau1996}, obtained by
relaxing the products $\mathcal{I}_{LM}$, $\mathcal{I}_{MM}$ of its own leading-order tensors along its path,
\begin{equation}
\mathcal{I}\leftarrow\mathcal{I}+\epsilon\,(\mathcal{I}^{(p)}-\mathcal{I}),\qquad \epsilon=\frac{\Delta t/\tau}{1+\Delta t/\tau},\qquad
\tau=1.5\,\Delta\,(\mathcal{I}_{LM}\mathcal{I}_{MM})^{-1/8},
\label{eq:lagrangian}
\end{equation}
and taking $C_p=\max(C,\,\mathcal{I}_{LM}/\mathcal{I}_{MM})$, starting from $C$ at birth. The excess $\nu_{T,p}=(C_p-C)\Delta^2|\mathsf{S}_p|$ is applied by particle strength
exchange (PSE)~\cite{degond1989,winckelmans1993} once per fluid step,
\begin{equation}
\left(\frac{\dd\ap}{\dd t}\right)_{T}=\sum_q\frac{c_L\,\nu_{pq}}{\sigma_{pq}^2}\,
\eta_{\sigma_{pq}}(|\xp-\vect{x}_q|)\,\bigl(V_p\vect\alpha_q-V_q\ap\bigr),\qquad
\nu_{pq}=\tfrac12(\nu_{T,p}+\nu_{T,q}),\quad \sigma_{pq}^2=\tfrac12(\sigma_p^2+\sigma_q^2),
\label{eq:pse}
\end{equation}
with $\eta$ a normalized Gaussian truncated at three widths and $c_L=2/m_2$, $m_2$ its second moment. In laminar or
young flow $C_p=C$ and the closure reduces to core spreading; where turbulence has developed, the exchange cancels
strength of opposite sign between neighbours.

The exchange is advanced by one explicit step. Its operator is similar to a symmetric negative semidefinite matrix
with columns summing to zero, so by Gershgorin's theorem its eigenvalues lie in $[-2\max_qD_q,0]$, with
$D_q=\sum_pc_L\nu_{pq}\eta_{\sigma_{pq}}V_p/\sigma_{pq}^2$ its diagonal. Where particles tile space,
$\sum_pV_p\eta_{\sigma_{pq}}\approx1$ and the usual bound on $\nu_T\Delta t/\sigma^2$ suffices, but where newly shed
particles crowd together the diagonal can be far larger. Each particle's exchanged viscosity is therefore limited to
$\nu_{T,p}\le1/(4\Delta t\hat D_p)$, $\hat D_p$ being its diagonal per unit viscosity, which keeps
$\Delta tD_q\le(1+\kappa)/8$ with $\kappa$ the largest ratio of neighbouring diagonals: no mode changes sign while
$\kappa<3$. The
limit acts on the operator, so the exchange stays antisymmetric.

The exchange's pairs conserve total vorticity exactly, but they move strength between points up to a support apart
and so change the linear impulse, which the diffusion they discretize conserves; with cores comparable to the rotor
radius this drains a wake's momentum. Its change of the strength rates is therefore corrected as the relaxation's is,
\cref{eq:conservative-relaxation}, so that the closure conserves total vorticity and linear impulse exactly.

\subsection{Relaxation}
\label{sec:relaxation}

The discrete field is not exactly solenoidal, and particle strengths drift out of alignment with the curl of the
reconstructed velocity. Following Pedrizzetti~\cite{pedrizzetti1992,winckelmans1993}, the strengths are relaxed
after each coupled step, here by projection onto the local vorticity direction,
\begin{equation}
\ap\leftarrow(1-\lambda)\ap+\lambda\,(\ap\cdot\hat{\wvec}_p)\,\hat{\wvec}_p,\qquad
\hat{\wvec}_p=\frac{\wvec_p^*}{|\wvec_p^*|},\quad \wvec_p^*=\nabla\times\uvec(\xp),
\label{eq:relaxation}
\end{equation}
with $\lambda=0.03$. The projection lets the misaligned part of each strength, which stems from unresolved
scales~\cite{cottet2000}, decay, whereas the blend towards $|\ap|\hat{\wvec}_p$ keeps $|\ap|$ and so converts
misaligned strength into aligned strength.

Neither form conserves linear impulse, and the rVPM correction $-\frac35(\vect{S}_p\cdot\hat\ap)\hat\ap$ of
\cref{eq:rvpm} changes it as well. With $\vect d_p$ the relaxation's change of particle $p$, we remove the net
changes of total vorticity and impulse by the least correction weighted by $|\vect d_p|$,
\begin{equation}
\vect d_p\leftarrow\vect d_p-|\vect d_p|\bigl(\vect a+\vect b\times(\xp-\bar{\vect x})\bigr),
\label{eq:conservative-relaxation}
\end{equation}
with $\vect a$, $\vect b$ and the weighted centroid $\bar{\vect x}$ such that $\sum\vect d_p=\vect 0$ and
$\sum\xp\times\vect d_p=\vect 0$. The free wake's strength rate in \cref{eq:rvpm} is corrected in the same way,
so that $\sum\dd\ap/\dd t$ and $\sum(\uvec_p\times\ap+\xp\times\dd\ap/\dd t)$ vanish. Both corrections are applied
over the whole wake, at $\order(N)$ cost and one global reduction; because the relaxation's changes are coherent
over a vortex, cancelling them locally would undo much of the realignment itself.

As a reference closure we also use a \emph{PSE variant}, which applies the molecular and dynamic eddy viscosity entirely
through \cref{eq:pse} with each particle's own core as filter width, $\Delta^2=12\sigma_p^2$, and relaxes the
strengths by the blend towards $|\ap|\hat{\wvec}_p$ with $\lambda=0.3$~\cite{alvarez2024}.
It appears in \cref{sec:ver-sfs,sec:val-wake}.

\subsection{Particle lifecycle}
\label{sec:lifecycle}

Particles are created by the rotors. Particles whose strength leaves
$[10^{-4},10]\,\alpha_{\mathrm{ref}}$, $\alpha_{\mathrm{ref}}$ a characteristic newborn strength, or whose core
leaves $[0.1\sigma_{s,\min},5\sigma_{s,\max}]$, $\sigma_s$ the source cores, are removed,
and a particle with a non-finite proposed state is removed and counted. A removed strength above the lower bound is handed to
nearby particles of the same rotor's wake with local maximum-entropy weights~\cite{arroyo2006}, which preserve its sum
and first moment and hence total vorticity and linear impulse. A run may also bound its memory by removing its oldest
particles once its store fills; this removal is not conservative and is used only in modest fidelity calculations with large numbers of drones such as \cref{sec:swarm}.

\subsection{Rotor models and wake coupling}
\label{sec:rotors}

A rotor of radius $R_r$ with $B$ blades spins at $\Omega$ about its axis $\vect n$. At a blade section of chord $c$
and twist $\vartheta$, the relative velocity projected on the plane normal to the span gives the speed $W$ and the
inflow angle $\phi$. With $c_l$ and $c_d$ taken from tabulated polars at the incidence $\vartheta-\phi$,
\begin{equation}
L'=\tfrac12\rho W^2c\,c_l,\qquad D'=\tfrac12\rho W^2c\,c_d,\qquad \Gamma=\tfrac12Wc\,c_l .
\label{eq:blade-element}
\end{equation}
Three representations are available (\cref{fig:rotor-models}). The \emph{performance-map disk} (PM) interpolates
$C_T$ and $C_Q$ from tables at the operating point given by the flow at a probe upstream of the disk and spreads the
thrust uniformly over equal-area annuli on $N_\theta$ rotating spokes of zero chord. The \emph{blade-element disk}
(BED) places on each of $N_r$ annuli $N_\theta$ rows with the blade's chord and weight $B/N_\theta$, and applies
\cref{eq:blade-element} with the inflow averaged around each annulus. The \emph{actuator line} (AL) resolves the $B$
blades on the same elements. All three sample the free wake and the background flow at their load points with the
particle solver; only the blade-element disk adds a tip loss, below.

\begin{figure}[t]
\centering
\includegraphics[width=\textwidth]{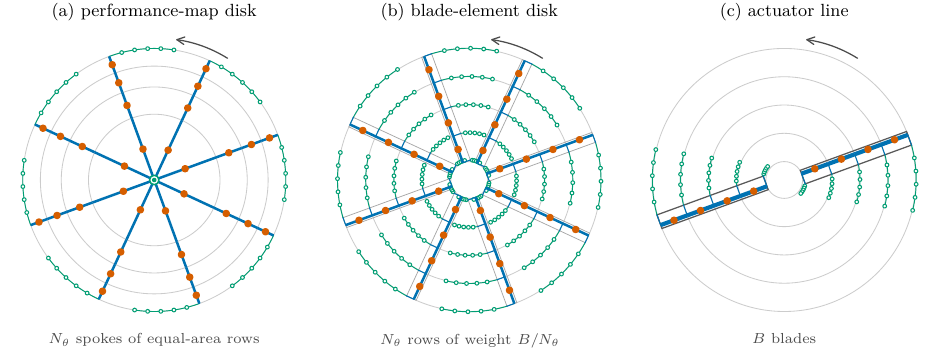}
\caption{The rotor representations under steady loading, viewed along the axis (rotation anticlockwise). Load points
(dots) carry bound circulation (thick lines) at the quarter chord; trailing vortices leave the row edges as particles
(circles). (a) Performance-map disk: uniform loading on equal-area annuli and $N_\theta$ zero-chord spokes, so
only tip trailers and a root vortex are shed. (b) Blade-element disk: $N_\theta$ rows of weight $B/N_\theta$ with
the blade's chord on elements of equal width. (c) Actuator line: the $B$ blades on the same elements.}
\label{fig:rotor-models}
\end{figure}

The free wake begins only where shed circulation becomes particles, at the trailing edge and at the end of the
step, so a lifting line that samples it misses the downwash of its own nearest trailing sheet, which is strongest
outboard. The actuator line adds it: for every section $j$, the vortex lines from the quarter chord to the trailing
edge, the trailing-edge path of the current step and, for the other blades, the bound segments induce
\begin{equation}
\vect{u}_{\mathrm{ns}}(\vect{x}_i)=\sum_j\sum_{s\in\mathcal{S}_j}\Gamma_s\,
\vect{K}(\vect{x}_i;\vect{a}_s,\vect{b}_s)\left(1-\mathrm{e}^{-d_{is}^2/\delta_j^2}\right),
\qquad \delta_j=\tfrac14c_j,
\label{eq:near-sheet}
\end{equation}
where $\vect K$ is the Biot--Savart kernel of a straight segment of unit circulation~\cite{katz2001} and $d_{is}$
the distance of $\vect{x}_i$ from the segment's line. The core of a quarter chord lies in the range of optimal
actuator-line smoothing widths, \numrange{0.14}{0.25}$\,c$~\cite{martineztossas2017}; it makes the sum independent of the
discretization and suppresses spanwise oscillations of the explicit coupling between sections. As in lifting-line
solvers~\cite{alvarez2022flowunsteady}, the circulation is relaxed, $\Gamma\leftarrow\frac12(\Gamma_{\mathrm{polar}}+\Gamma_{n-1})$,
which leaves steady states unchanged.

The blade-element disk spreads each blade over $N_\theta$ rows and samples the annulus average, which underestimates the induction
near the tip: there a blade sees more induction than the average, from its own trailing sheet just behind it.
The disk restores it with Prandtl's tip loss in its inflow form~\cite{glauert1935,leishman2006}. Prandtl's factor
\begin{equation}
F=\frac2\pi\arccos\exp\left(-\frac B2\,\frac{R_r-r}{r\sin\phi}\right),\qquad \tan\phi=\frac{U}{V_t},
\label{eq:prandtl}
\end{equation}
relates the induction averaged over the annulus at radius $r$ to the induction $v_b$ at the blade, and the annulus
momentum balance $Fv_bU=B\Omega\Gamma/(4\pi)$ gives $v_b$ from the ring's circulation. Each ring adds the excess
$(1-F)v_b$ to its averaged through-flow $\bar U$, so that the through-flow at the blade is $U=\bar U+(1-F)v_b$, with
$V_t$ the tangential speed; $\phi$ and the excess are solved together by fixed-point iteration. Only the rotor's own
induction is raised, so external inflow, such as another vehicle's wake, is not amplified. Raising the inflow rather
than scaling the lift lowers the incidence of the outer sections, so that their lift and profile drag fall together.

\paragraph{Connected shedding and impulse}
All representations shed connected circulation, adapting the circulation-based shedding
of~\cite{alvarez2022flowunsteady} to disks. Each row carries bound circulation at its quarter chord: the
Kutta--Joukowski value of \cref{eq:blade-element} or, for a map row at radius $r$ with width $\Delta r$ and thrust
share $f_z$, $\Gamma=f_z/(\rho\Omega r\Delta r)$. Spanwise differences of $\Gamma$ are shed as trailing legs along
the path that each row edge's trailing end travels over the step, and temporal changes as spanwise legs, each
becoming particles of total strength $\Gamma\vect\ell$ with the source core $\sigma_s$; coincident edges of
neighbouring rows are merged. An actuator-line blade sheds the straight chord of that path, placed by Gauss
quadrature, which requires the step to resolve the rotation: over a rotation by $\beta$ the chord encloses only
$\sin\beta/\beta$ of the area the blade sweeps, so actuator lines take at least 18 steps per revolution, where
the chord carries $98\%$ of it. The disks shed along the arc itself, whole turns included, and gather the arcs of each
ring edge on azimuth bins and time layers of a quarter core. The number of particles born per step is then bounded, and a step
that resolves the rotation gives one particle per bin, as the quadrature would. Since a closed loop of
circulation $\Gamma$ has impulse $\rho\Gamma\vect A$, $\vect A$ its vector area~\cite{cottet2000}, the vortex system
born in a step carries the impulse of the loop swept by the bound circulation, the step's lift impulse. The discrete
balance
\begin{equation}
\vect{I}_{\mathrm{newborn}}+\Delta\vect{I}_{\mathrm{bound}}-\vect{I}_{\mathrm{transport}}
=\vect{I}_{\mathrm{swept}}+\vect{I}_{\mathrm{unsteady}}
\label{eq:birth-impulse}
\end{equation}
holds per rotor and step to a relative residual of $10^{-15}$ in double and $5\times10^{-7}$ in single precision for
the actuator line; the disks' gathering on bins departs from their arcs' impulse by less than $10^{-9}$ in double
precision when the step resolves the rotation, and by less than $10^{-3}$ over steps of several turns.
Profile drag acts on the vehicle only, and a body-force representation of the lift would duplicate the bound
vorticity~\cite{meyerforsting2019}. The ratio
\begin{equation}
\Pi(t)=\frac{-\rho\,\vect{n}\cdot\vect{I}_{\mathrm{wake}}(t)}{\int_0^t T\,\dd t'},\qquad
\vect{I}_{\mathrm{wake}}=\frac12\sum_p\xp\times\ap,
\label{eq:impulse-ratio}
\end{equation}
of the wake's impulse to the delivered thrust impulse is one at birth and measures how well the wake keeps the
momentum it receives.

\subsection{Vehicle dynamics and partitioned time integration}
\label{sec:coupling}

Each vehicle is a rigid body governed by the Newton--Euler equations with the rotors' gyroscopic moments and a
quaternion attitude, loaded by its rotors and by a parasitic drag $-\frac12\rho C_DA\,|\vect{u}_r|\,\vect{u}_r$, with
$\vect{u}_r$ the body's velocity relative to the flow sampled at its centre; the airframe is not otherwise resolved.
Vehicles follow prescribed motor speeds or position and heading targets through a cascaded proportional--derivative
controller~\cite{mahony2012}, whose wrench is allocated to the motors by bounded least squares on the rotors' current
aerodynamic wrenches, so that the allocation responds to wake interactions; the motor speeds may follow their commands
with a first-order lag, and the gains are then placed for the lagged loop. An outer step $\Delta t_o$ carries loads,
control, shedding and the vehicles, and for the actuator line resolves the blade passage; the vehicles advance by the
strong-stability-preserving third-order scheme of Shu and
Osher~\cite{shu1988,gottlieb2001}. The wake advances by $N_f$ substeps of the low-storage third-order Runge--Kutta
scheme of Williamson~\cite{williamson1980} on positions, strengths and $\sigma^2$, with bound circulation and
vehicle positions held at their accepted values; the coupling is therefore first order in $\Delta t_o$, while
refining the substeps isolates the wake's own temporal error.

\subsection{Particle solver and GPU implementation}
\label{sec:solver}

Every Runge--Kutta stage needs $\uvec$ and $\Jt$ at all particles. The evaluator is a treecode~\cite{barnes1986}
with solid-harmonic expansions of degree $P=16$~\cite{greengard1987,dehnen2014}, batched over target leaves, on a
k-d tree of Morton-ordered particles~\cite{morton1966,karras2012} with leaves of at most 128 particles.
A source node is evaluated by its expansion for a target leaf if (i) $(h_t+h_s)/|\vect c_t-\vect c_s|\le\theta=0.45$,
(ii) the gap between the two cubes exceeds $R\sigma_{s,\max}$, and (iii) the target cube lies outside the source
cube's convergence sphere. By \cref{prop:exterior}, test (ii) makes the far field exact up to the truncation of the
series, which decreases as $\varrho^{P+1}$ with the convergence ratio $\varrho\le\sqrt3\theta$. In single precision
each accepted node uses the smallest degree in $\{8,12,16\}$ whose truncation term does not exceed that of a node
with $\varrho=0.65$ at full degree. Other nodes are summed directly with the compact core, so the near work grows with
the local density, not with $N$, and one evaluation costs $\order(NP^2\log N+N\bar n)$, $\bar n$ the mean number of
neighbours within the core support. Interaction lists are never truncated.

Swarmlet is written in Julia~\cite{bezanson2017} with portable kernels (KernelAbstractions~\cite{kernelabstractions},
CUDA.jl~\cite{besard2019}), and all state, including vehicles, controllers and motor allocation, stays resident on the
GPU. Double precision serves as a reference and single precision, with near-field sums accumulated in
double~\cite{higham2002}, for production; both precisions use the same code and differ only in number representation. On several MPI processes~\cite{byrne2021}, particles are partitioned along the Morton order and each rank
receives only the near sources and accepted expansions that its targets need, a locally essential
tree~\cite{warren1993}; no rank holds the global particle set. Two- and four-rank runs agree with one-rank runs to the partition-dependent far-field
error. On one NVIDIA RTX~4080, a right-hand-side evaluation costs \SIrange{100}{140}{ns} per particle from $6.6\times10^4$
to $4.2\times10^6$ uniformly distributed particles, and the interaction kernel, $82\%$ of a coupled step, runs at about $40\%$ of the
single-precision peak. \Cref{tab:parameters} lists the wake parameters, which are the same for every rotor case of
\cref{sec:validation}.

\begin{table}[t]
\centering
\caption{Parameters of the wake method, common to the rotor cases of \cref{sec:validation}; the swarm of
\cref{sec:swarm} uses a coarser setting. Rotor geometry, source core, time step and precision are given with each
case.}
\label{tab:parameters}
\footnotesize
\begin{tabular}{@{}ll@{}}
\toprule
Quantity & Value \\
\midrule
Core & compact, support $R=4\sigma$; rVPM $f=0$, $g=1/5$; transposed stretching \\
Time integration & Williamson low-storage RK3 (wake); SSP-RK3 (vehicles) \\
Far field & $P=16$ (single precision: adaptive $P'\in\{8,12,16\}$), $\theta=0.45$, gap $>R\sigma_{s,\max}$ \\
Subfilter model & dynamic, $\Delta^2=12\sigma_{s,\max}^2$; global $C$ by core spreading with $\nu$; Lagrangian excess by PSE, \\
 & limited to $\nu_{T,p}\le1/(4\Delta t\hat D_p)$, its changes of $\sum\ap$ and impulse removed globally \\
Relaxation & projection, $\lambda=0.03$ per outer step; changes of $\sum\ap$ and impulse removed globally \\
Lifecycle & strength window $[10^{-4},10]\,\alpha_{\mathrm{ref}}$, core window $[0.1\sigma_{s,\min},5\sigma_{s,\max}]$; \\
 & removed strength redistributed \\
Actuator line & near sheet with core $c/4$, circulation relaxation $1/2$; $\ge18$ steps per revolution \\
Blade-element disk & Prandtl's tip loss on the annulus inflow, \cref{eq:prandtl} \\
Disk shedding & swept arcs on bins and time layers of $\sigma_s/4$ \\
\bottomrule
\end{tabular}
\end{table}

\section{Verification}
\label{sec:verification}

\subsection{Accuracy of the particle solver}
\label{sec:ver-solver}

The solver departs from direct summation with the same core only by far-field truncation and rounding. We evaluate
$\uvec$ and $\Jt$ at \num{4096} random targets among $N=\num{32768}$ particles in three distributions (randomly placed
vortex rings, Gaussian clusters and a uniform cube) against a direct double-precision sum (\cref{fig:field-accuracy}). In double precision the relative $L^2$ errors fall geometrically with
the expansion degree, by a factor of about $0.53$ per degree, the rate set by the convergence ratios that
$\theta=0.45$ bounds by $0.78$. At $P=16$ the velocity errors are \numrange{1.5e-7}{4.2e-7} and the gradient
errors \numrange{3.5e-7}{4.3e-7}. The single-precision production solver, with adaptive degree and
double-precision accumulation, gives \numrange{2.9e-7}{5.0e-7}, and raising its degree to 24 only lowers them to
the single-precision rounding level, \numrange{1.1e-7}{3.3e-7}: at $P=16$, truncation and rounding are of the
same order.

\begin{figure}[t]
\centering
\includegraphics[width=\textwidth]{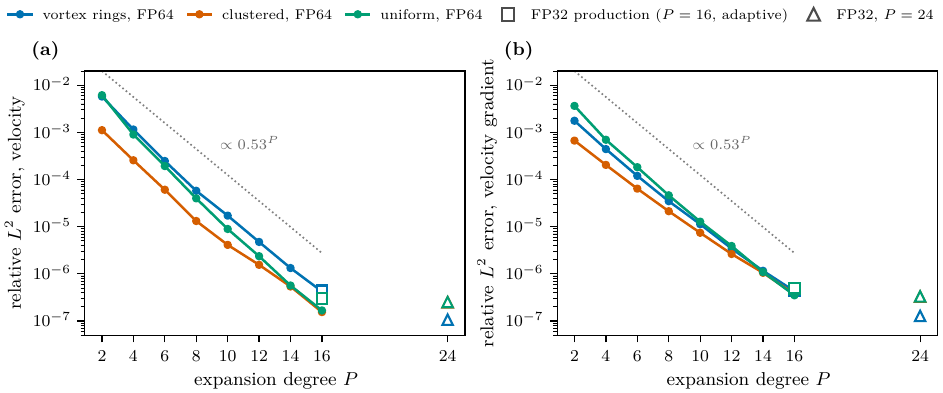}
\caption{Relative $L^2$ error of (a) the velocity and (b) the velocity gradient against direct summation with the
same core, $N=\num{32768}$, three distributions. Lines: double precision against the expansion degree $P$. Open symbols:
the single-precision production solver at $P=16$ and at $P=24$.}
\label{fig:field-accuracy}
\end{figure}

\subsection{Vortex rings: convergence in space and time}
\label{sec:ver-ring}

A ring of radius $R_0=\SI{0.5}{m}$, circulation $\Gamma=\SI{1}{m^2.s^{-1}}$ and core $\sigma=\SI{0.05}{m}$ is
represented by particles on a circle with inviscid rVPM. As the spacing $h$ is refined, the initial translation speed
converges rapidly once the cores overlap (\cref{fig:ring-convergence}a): its difference from the finest
discretization is $3.5\times10^{-3}$ at $h/\sigma=2$, $5.3\times10^{-6}$ at $h/\sigma=1$ and $1.2\times10^{-7}$ at
$h/\sigma=0.5$, and then stays at the solver's truncation level. The converged speed is $0.44\%$ above that of a
continuous Gaussian-core filament computed by quadrature, in both precisions and at all resolutions; this is a
property of the compact core, whose smaller second moment makes a thin ring faster~\cite{saffman1970}. For temporal
convergence, two coaxial rings leapfrog for \SI{1}{s} with time steps from 40 to \SI{2.5}{ms}
(\cref{fig:ring-convergence}b). The observed orders of the position error are $2.90$, $2.97$ and $2.99$, confirming
third order for the coupled rVPM system including the evolving cores, and single precision changes the positions by
$10^{-6}R_0$ and the impulse by $9\times10^{-8}$.

\begin{figure}[t]
\centering
\includegraphics[width=\textwidth]{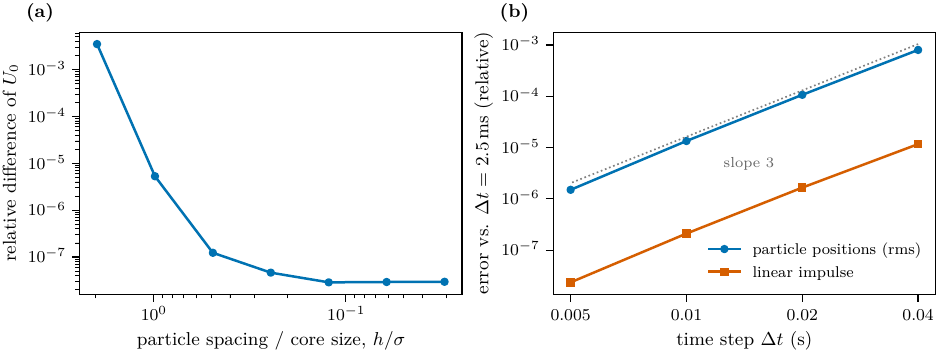}
\caption{Convergence on vortex rings (inviscid rVPM, double precision). (a) Initial translation speed against the
spacing $h$, as the difference from $h/\sigma=0.015$. (b) Leapfrogging ring pair at $t=\SI{1}{s}$: root-mean-square
position error relative to the ring radius and relative impulse error against the time step, measured against
$\Delta t=\SI{2.5}{ms}$.}
\label{fig:ring-convergence}
\end{figure}

\subsection{Subfilter closure on canonical rings}
\label{sec:ver-sfs}

An LES closure should leave a laminar ring nearly untouched. \Cref{fig:sfs-rings}(a) follows a ring with
$\Gamma/\nu=1500$ at two resolutions for \SI{3}{s} in single precision. With the
closure of \cref{sec:sfs}, without relaxation, the ring keeps $99.90\%$ of its linear impulse at both resolutions,
and with the PSE variant $98.5\%$, whose variable-viscosity PSE is not corrected for impulse. For reference we also show our
implementation, on the compact core, of a strength-dissipation dynamic procedure of the kind used in earlier rVPM
work~\cite{alvarez2024}, which removes strength along the stretching direction with a per-particle coefficient. It
keeps $54\%$ and $51\%$ and does not improve with refinement, because its per-particle coefficient is dominated by
noise; this characterizes our implementation, not necessarily the original. The dynamic coefficient of the closure
is small but not zero on this ring (\numrange{0.011}{0.020}), so the mean core grows to \SI{0.103}{m} instead of the
\SI{0.081}{m} of molecular diffusion alone, and the ring travels $14\%$ less far than with PSE.
\Cref{fig:sfs-rings}(b) shows a leapfrogging pair with the complete closure, including relaxation: its impulse is
unchanged to four digits at \SI{6}{s} and \SI{10}{s}, against $97.5\%$ and $95.6\%$ of the initial with the PSE
variant.

\begin{figure}[t]
\centering
\includegraphics[width=\textwidth]{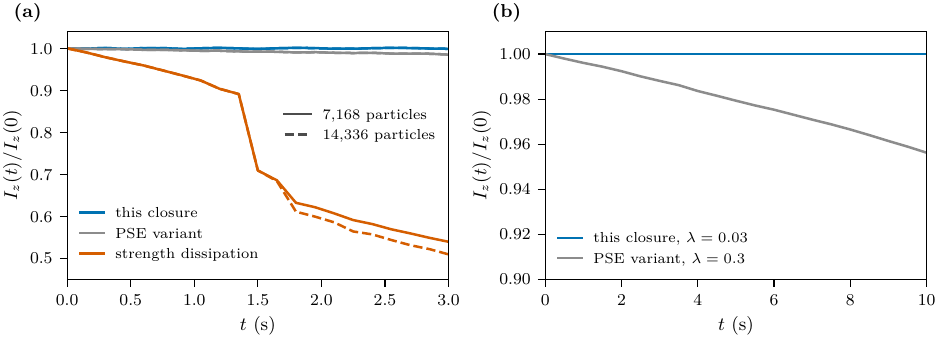}
\caption{Linear impulse $I_z$ of canonical rings (single precision, $\Gamma/\nu=1500$). (a) A laminar ring with the
closure of \cref{sec:sfs}, the PSE variant and a strength-dissipation dynamic model, each at two resolutions. (b) A
leapfrogging ring pair with the closure and relaxation by projection ($\lambda=0.03$) and with the PSE variant.}
\label{fig:sfs-rings}
\end{figure}

\section{Validation}
\label{sec:validation}

All runs use the wake method of \cref{tab:parameters} in single precision and start from an empty wake with the
rotors held in place. Loads fluctuate by several percent between revolutions, so we report revolution means over
declared windows together with the two halves of each window.

\subsection{DJI 9443 propeller in hover}
\label{sec:val-dji}

Zawodny et al.\ measured \SI{2.317}{N} for the two-blade DJI~9443 propeller (diameter \SI{0.24}{m}) at
\SI{5400}{rpm}~\cite{zawodny2016}. The model uses the blade geometry and section polars of FLOWUnsteady's hover
example~\cite{alvarez2022flowunsteady}, and the blade-element disk adds the tip loss of \cref{eq:prandtl}; nothing
is fitted to the measurement. The rotor has 8 radial elements, a source core $\sigma_s=0.1R_r$, 72 steps per revolution
with four fluid substeps, and runs for 20 revolutions.

\begin{figure}[t]
\centering
\includegraphics[width=\textwidth]{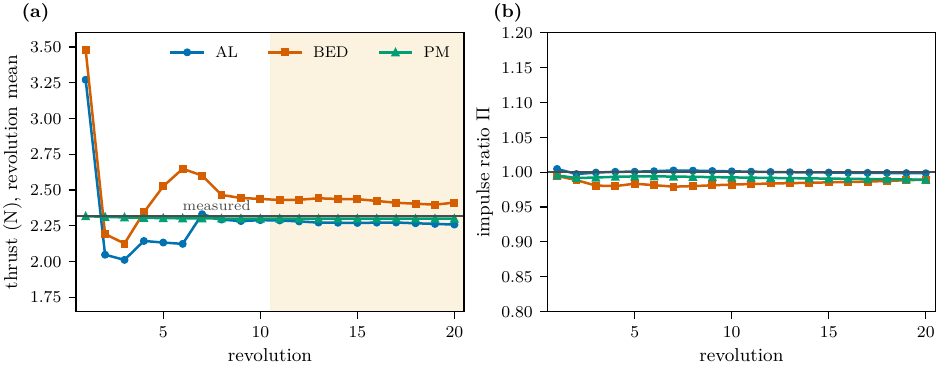}
\caption{DJI~9443 hover at \SI{5400}{rpm}. (a) Revolution-mean thrust of the actuator line (AL), blade-element disk
(BED) and performance-map disk (PM); the line is the measurement~\cite{zawodny2016} and the band the averaging window.
(b) Ratio $\Pi$ of the wake's linear impulse to the delivered thrust impulse, \cref{eq:impulse-ratio}.}
\label{fig:dji}
\end{figure}

\begin{table}[t]
\centering
\caption{DJI~9443 at \SI{5400}{rpm}: mean thrust over revolutions \numrange{11}{20} against the measured
\SI{2.317}{N}~\cite{zawodny2016}, with the two halves of the window, the range of the impulse ratio $\Pi$ over
20 revolutions, and the particle count at revolution 20.}
\label{tab:dji}
\small
\setlength{\tabcolsep}{5pt}
\begin{tabular}{@{}lccccc@{}}
\toprule
Rotor model & Thrust (N) & Difference & Halves & $\Pi$ & Particles \\
\midrule
Actuator line & 2.271 & $-1.99\%$ & $-1.78\%$, $-2.19\%$ & \numrange{0.997}{1.005} & \num{181399}\\
Blade-element disk & 2.422 & $+4.54\%$ & $+5.10\%$, $+3.98\%$ & \numrange{0.979}{0.995} & \num{291233}\\
Performance-map disk$^{a}$ & 2.299 & $-0.77\%$ & $-0.77\%$, $-0.77\%$ & \numrange{0.989}{0.995} & \num{67484}\\
\bottomrule
\multicolumn{6}{@{}l}{\footnotesize $^{a}$ The thrust is the map's input, not an independent prediction.}
\end{tabular}
\end{table}

The actuator line predicts the thrust within $2\%$ (\cref{fig:dji,tab:dji}); sampling only the free wake, with that model's
empirical hub- and tip-loss factors in place of its near sheet, it gives $-2.36\%$. The blade-element disk, which averages the
inflow around each annulus, is $4.5\%$ high; without a tip loss model it is $13\%$ high, and with the same hub- and tip-loss
factors applied to the lift $4.6\%$ low. The performance-map disk returns its map, computed by blade-element momentum
theory. All three wakes keep the momentum they receive: $\Pi$ starts at one, as \cref{eq:birth-impulse} requires,
and stays between $0.97$ and $1.01$, with a lateral impulse below $0.1\%$ of the axial. Both the connected shedding and the closure
are necessary for this: when the disks deposit their lift as the curl of a regularized body force instead, the map
disk's $\Pi$ reaches $2.75$ within five revolutions, and without an SFS model the actuator line's revolution-mean
thrust ranges from $-47\%$ to $+16\%$.

\subsection{Crazyflie propeller at three speeds}
\label{sec:val-crazyflie}

The UIUC database~\cite{uiucpropdb,brandt2011,deters2014} contains the static thrust and power of two specimens of
the \SI{46.99}{mm} propeller of the Crazyflie nano-quadrotor, which differ from their mean by $\pm$\numrange{2.5}{2.7}\%. No
measured polars exist for this thin, bent blade at chord Reynolds numbers of order $10^4$. The polars are built from
McArthur's cambered-plate measurements~\cite{mcarthur2008}, with the lift jump of a laminar separation bubble
replaced by the attached-flow lift line, and scaled by
\begin{equation}
k_l(\mathrm{Re})=0.736\left(\frac{\mathrm{Re}}{2\times10^4}\right)^{0.21},\qquad
k_d(\mathrm{Re})=1.13\left(\frac{\mathrm{Re}}{2\times10^4}\right)^{-0.15}
\label{eq:polar-fit}
\end{equation}
for lift and drag, held constant outside $4\times10^3\le\mathrm{Re}\le4\times10^4$. These factors were calibrated
on the measurements, as a least-squares fit (four parameters for six values) of the scales that reproduce the
measured thrust and torque at each speed. Their trends are those of thin plates at decreasing Reynolds number, which
lose lift and gain friction drag. The rotor has 24 radial elements, $\sigma_s=0.075R_r$ and 72 steps per revolution,
and each run covers 30 revolutions.

\begin{table}[t]
\centering
\caption{Static thrust and shaft torque of the \SI{46.99}{mm} Crazyflie propeller (actuator line), mean over
revolutions \numrange{9}{30}, against the mean of two UIUC specimens, with the thrust in the two halves of the window. The section
polars were calibrated on these points.}
\label{tab:crazyflie}
\footnotesize
\setlength{\tabcolsep}{4pt}
\begin{tabular}{@{}rcccccccc@{}}
\toprule
& \multicolumn{4}{c}{Thrust (mN)} & \multicolumn{3}{c}{Torque (\si{\micro N.m})} & \\
\cmidrule(lr){2-5}\cmidrule(lr){6-8}
rpm & Sim. & Meas. & Difference & Halves & Sim. & Meas. & Difference & Particles\\
\midrule
\num{10000} & 22.99 & 23.73 & $-3.14\%$ & $-2.21\%$, $-4.07\%$ & 149.2 & 147.1 & $+1.45\%$ & \num{402271}\\
\num{16923} & 74.24 & 75.79 & $-2.05\%$ & $-2.37\%$, $-1.74\%$ & 389.9 & 384.3 & $+1.45\%$ & \num{409783}\\
\num{19000} & 95.52 & 96.69 & $-1.22\%$ & $-1.38\%$, $-1.05\%$ & 482.6 & 482.9 & $-0.05\%$ & \num{405085}\\
\bottomrule
\end{tabular}
\end{table}

Thrust and torque agree with the specimen mean within $3.2\%$ at all three speeds (\cref{tab:crazyflie}). At
\num{16923} and \SI{19000}{rpm} this is inside the specimen spread; at \SI{10000}{rpm} the thrust is $0.6$ points
below it and falls over the window, from $-2.2\%$ in the first half to $-4.1\%$ in the second. No particle was
removed as non-finite. The near sheet governs the outer blade: at
\SI{19000}{rpm} the sections outboard of $0.7R_r$ operate at \SIrange{5.8}{9.2}{\degree} and the circulation peaks at $0.88R_r$,
close to a blade-element momentum solution with Prandtl's tip loss (\SIrange{6}{7}{\degree}).
Sampling only the free wake, the same actuator line sees its axial inflow fall from \SI{3.8}{m.s^{-1}} at $0.75R_r$
to \SI{0.7}{m.s^{-1}} at the tip, its outer sections operate beyond the lift maximum at \SIrange{9}{14}{\degree}, and with
the polars scaled to the measured thrust its torque exceeds the measurement by \numrange{10}{17}\%.

\subsection{Near wake of a hovering quadrotor}
\label{sec:val-wake}

Kiran et al.\ measured the velocity below a hovering Crazyflie quadrotor by planar particle image velocimetry
(PIV)~\cite{kiran2024,kiran2025,kiran2025data}. We compare the axial velocity across the plane through an adjacent
rotor pair at every measured station from $z/l=2$ to $12$ below the rotors, $l=\SI{32.5}{mm}$ being half the rotor
spacing, normalized by the induced velocity $U_i=\sqrt{mg/(2\rho A)}=\SI{8.75}{m.s^{-1}}$, with $mg$ the vehicle's
weight and $A$ the area of one rotor disk, as defined by the authors' processing~\cite{kiran2025data}. The model is a \SI{32}{g} x-layout quadrotor with the blade of
\cref{sec:val-crazyflie} scaled to its \SI{45}{mm} rotors, the same polars and numerics, and all rotors at
\SI{18777}{rpm}, where the measured thrust coefficient at the same chord Reynolds number carries the weight; the
rotor speed was not measured. The run covers 64 revolutions from rest, and the velocity is averaged over revolutions
\numrange{53}{64}, the four adjacent rotor pairs and both sides of each pair's symmetry plane, with the measurement folded alike
and nothing shifted or rescaled. Over the window the rotors deliver $2.3\%$ less than the weight.

\begin{figure}[t]
\centering
\includegraphics[width=\textwidth]{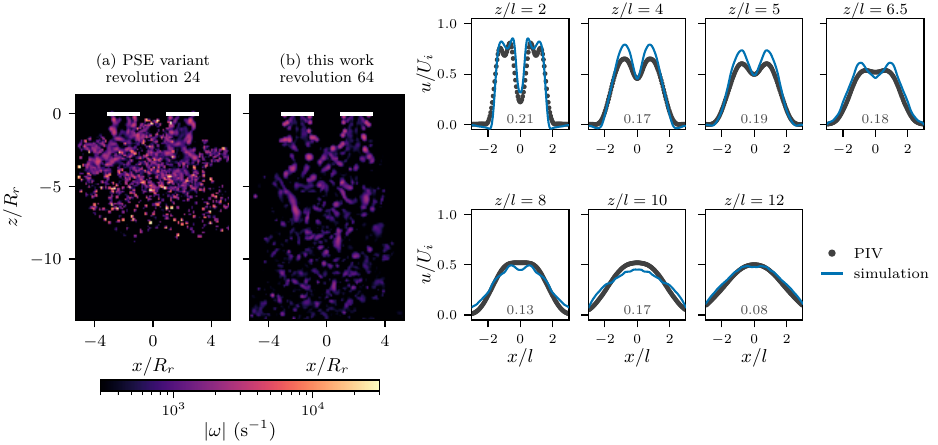}
\caption{Hovering Crazyflie-type quadrotor at \SI{18777}{rpm}. Vorticity magnitude on the plane through two opposite
rotors, reconstructed from the particles: (a) the PSE variant after 24 revolutions and (b) the closure of
\cref{sec:sfs} after 64; $R_r$ is the rotor radius. Right: axial velocity across an adjacent rotor pair at each PIV
station~\cite{kiran2025data} and the simulation averaged over revolutions \numrange{53}{64}, with the relative $L^2$ difference
over $|x/l|\le3$.}
\label{fig:hover-wake}
\end{figure}

The relative $L^2$ difference is $0.21$, $0.17$, $0.19$, $0.18$, $0.13$, $0.17$ and $0.08$ at $z/l=2$, 4, 5, 6.5, 8,
10 and 12 (\cref{fig:hover-wake}). Close to the rotors the simulated jets have the measured structure: at $z/l=2$
they peak at $0.86U_i$ against $0.81U_i$, each with a dip on the rotor axis, the velocity midway between the rotors is
$2.76$ against \SI{1.96}{m.s^{-1}}, and there is no flow outside the jets. The measured dip is about \SI{10}{mm} wide
and matches the motors below the rotors, which are not modelled; the simulated dip comes from the inboard loading.
This near field depends on the near sheet. Sampling only the free wake, with polars scaled to the thrust, the
actuator line loads the blade inboard; its jets at $z/l=2$ then peak at $1.07U_i$ with \SI{0.53}{m.s^{-1}} between the rotors, and the difference there is
$0.36$. The jets merge into one column by $z/l\approx8$, against about 6.5 in the measurement. From $z/l=4$ to 12 the
centreline speed is within $14\%$ of the measurement, and the half-width matches at $z/l=4$ and 5 (\SI{59}{mm} against
\SIrange{57}{59}{mm}) but grows faster below, to \SIrange{77}{78}{mm} against \SIrange{62}{66}{mm} at $z/l=\numrange{10}{12}$. Below
$z/l\approx6$ the column is surrounded by a co-flow of \SIrange{0.3}{1.5}{m.s^{-1}}, stronger than the measured
\SIrange{0.05}{0.64}{m.s^{-1}}. It is the remnant of the starting vortex of a simulation started from rest, whereas the
measurement is a steady state: \numrange{17}{27}\% of $\sum_p|\ap|$ lies more than $3l$ from the axis.

\Cref{fig:hover-wake}(a) shows why the filter width must not follow the particle cores. In the PSE variant the aged
wake below about four radii breaks into isolated particles whose vorticity reaches \SI{e6}{s^{-1}}, \numrange{200}{500}
times the rotor-plane value: stretched cores shrink to the core floor while their eddy viscosity and their PSE
neighbours vanish. That wake loses its momentum ($\Pi=0.37$ at revolution 24) and carries almost no flow below
$z/l\approx3$, with relative differences of \numrange{0.80}{0.97} at $z/l=\numrange{4}{8}$.

\section{Demonstration: a formation of 100 quadrotors}
\label{sec:swarm}

To illustrate the coupled method at the scale of a formation, 100 quadrotors fly together in closed loop
(\cref{fig:swarm}). Each is the \SI{32}{g} vehicle of \cref{sec:val-wake} with the inertia identified by
F{\"o}rster~\cite{forster2015} and four performance-map disks. The map takes the measured static thrust and torque
coefficients of the propeller~\cite{uiucpropdb} at the chord Reynolds number of \cref{sec:val-wake}, and their
variation with advance ratio from blade-element momentum theory on the calibrated polars of \cref{sec:val-crazyflie}.
The motors follow their commands with the first-order lag of \SI{65}{ms} identified in~\cite{forster2015}, the
controller's gains are placed for the lagged loop, and each body carries a parasitic drag with
$C_DA=\SI{3e-4}{m^2}$. The controller holds every vehicle on its slot of a sphere of radius $R_s=\SI{1.5}{m}$, with
nearest neighbours \SI{0.46}{m} apart, for \SI{1}{s}, and then moves the formation \SI{3}{m} along $x$ on
minimum-jerk paths~\cite{flash1985} of \SI{3}{s}.

A formation calls for a coarser setting than the cases of \cref{sec:validation}. The far field uses expansions of
uniform degree $P=8$ (field error about $4\times10^{-5}$), the disks gather their swept arcs on bins and time layers of
half a core, the source core is $\sigma_s=0.5R_r$, and the outer step of \SI{1.07}{ms} spans a third of a hover
revolution, with one fluid substep. For an isolated vehicle held at its hover speed, the thrust at this setting differs
by $0.4\%$ from that at $\sigma_s=0.075R_r$ and twelve steps per revolution (the map's thrust depends on the wake only
through the inflow at its probe), and the wake keeps $0.99$ of the delivered impulse over \SI{0.4}{s}. To bound memory,
wake older than about \SI{0.45}{s} is removed once the store of $7.2\times10^6$ particles fills, a removal that is not
conservative (\cref{sec:lifecycle}). This particle limit was determined by the available GPU memory of an NVIDIA 4080 device used for this calculation.
\begin{figure}[t]
\centering
\includegraphics[width=\textwidth]{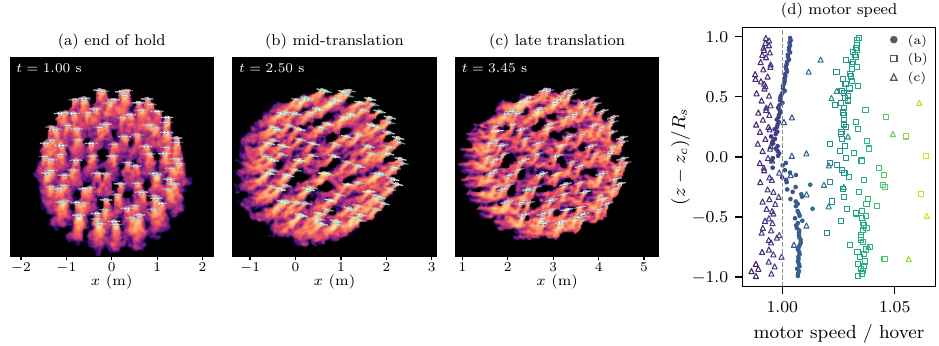}
\caption{Formation of 100 quadrotors on performance-map disks. (a--c) Wake strength $|\ap|$ projected on a
vertical plane (side view, logarithmic scale) at the end of the hold, midway through the translation and late in it,
each panel centred on the formation; the vehicles are drawn at twice their size in their attitude, with their rotor
disks coloured by motor speed as in (d). (d) Motor speed of each vehicle relative to an isolated vehicle in hover,
averaged over the last \SI{0.4}{s} of the hold and over the \SI{0.5}{s} before (b) and (c), against its height in the
sphere.}
\label{fig:swarm}
\end{figure}

At the end of the hold every vehicle stands above a short column of its own wake, and the wakes of the vehicles just
above descend onto those below them (\cref{fig:swarm}a). The lower hemisphere therefore needs a median $0.7\%$ more
motor speed than an isolated vehicle, and up to $1.4\%$, against $0.1\%$ in the upper hemisphere (\cref{fig:swarm}d).
Midway through the translation, at the peak speed of \SI{1.9}{m/s}, the wakes stream back and down behind the vehicles
(\cref{fig:swarm}b), which need \numrange{2}{6}\% more motor speed to tilt and accelerate; eight of the ten vehicles above
$4\%$ fly at or below the equator, where the wakes of the vehicles ahead and above pass. Late in the translation, as
the formation brakes through \SI{0.7}{m/s} (\cref{fig:swarm}c), the median vehicle runs \numrange{0.5}{0.7}\% below its
hover speed, its rotors aided by the oncoming flow, while four vehicles in the wakes of their neighbours still need up
to $6\%$ more.

\section{Discussion and limitations}
\label{sec:discussion}

The results indicate that the treatment of aged vorticity, rather than the accuracy of the far field, decides whether
a particle wake keeps its rotor's momentum. With a filter width that follows the particle cores, stretched cores
collapse with their eddy viscosity and a hovering quadrotor's wake retains about a third of the delivered impulse.
With a fixed filter width, core spreading and a Lagrangian exchange of the local excess, the impulse ratio of every
rotor wake stays between $0.97$ and $1.01$ and canonical rings keep at least $99.9\%$ of their impulse. The global
corrections contribute to this: the exchange, the relaxation and the rVPM correction term all change the wake's
impulse, and removing only their net change preserves what each does locally.

The validation has clear limits. The small rotor's polars were calibrated, and the DJI rotor's polars are modelled; each comparison concerns the mean loads of a hovering rotor. The
blade-element disk relies on the classical tip loss and over-predicts the DJI rotor's thrust by $4.5\%$. The vehicle's
motors, arms and body are not represented, although the measured jets show some of their character. The formation of \cref{sec:swarm} is a 
demonstration rather than a validation.

\section{Conclusions}
\label{sec:conclusions}

Swarmlet combines a meshless reformulated vortex particle method with rotor models of three fidelities, rigid-body
vehicles and flight control in a single GPU-resident code. A compact-support core makes the multipole far field exact
up to truncation and gives every near-field operator a finite range. A dynamic eddy viscosity on a fixed filter width,
applied by core spreading with a Lagrangian exchange of its local excess, and a relaxation by projection, both with
global conservative corrections, close the filtered equations without core collapse, and the wakes keep the lift
impulse that connected bound circulation delivers to them. The particle solver agrees with direct summation to a few parts in
$10^7$ and the particle equations converge at third order in time. Without calibration the actuator line predicts the
hover thrust of a reference propeller within $2\%$; with calibrated polars it reproduces the thrust and torque of a
small propeller within $3.2\%$ across a factor of two in speed, and the wake of a hovering quadrotor within relative
differences of \numrange{0.08}{0.21} from one to six rotor spacings below the rotors. At a coarser setting, the same
method flies a closed-loop formation of 100 quadrotors on one GPU.

\section*{Code and data availability}
\label{sec:availability}

The experimental data used in this work is those of the cited
sources~\cite{zawodny2016,uiucpropdb,mcarthur2008,kiran2025data}.

\section*{CRediT authorship contribution statement}

\textbf{Wyatt Horne:} Conceptualization, Methodology, Software, Validation, Formal analysis,
Investigation, Data curation, Writing -- original draft, Writing -- review \& editing, Visualization.

\section*{Declaration of competing interest}

The author declares no known competing financial interests or personal relationships that could
have appeared to influence the work reported in this paper.

\section*{Declaration of generative AI and AI-assisted technologies in the writing process}

During the preparation of this work the author used generative AI tools to assist with implementation,
testing, literature checks and the preparation of the manuscript.

\bibliographystyle{elsarticle-num}
\bibliography{references}

\end{document}